\documentclass[a4paper, 12pt]{amsart}

\usepackage[T1]{fontenc}
\usepackage[ngerman,english]{babel}
\usepackage{amsmath}
\usepackage{amsthm}
\usepackage{amssymb}
\usepackage{amsfonts}

\usepackage{tabularx}
\usepackage{multirow}

\DeclareMathAlphabet{\mathpzc}{OT1}{pzc}{m}{it}

\usepackage{tikz}
\usetikzlibrary{positioning,matrix}
\usepackage{MnSymbol}
\usepackage{nicefrac}
\usepackage{paralist}

\theoremstyle{plain}
\newtheorem{theorem}{\bf{\textsc{Theorem}}}[section]

\newtheorem{theoremAlpha}{\bf{\textsc{Theorem}}}

\newtheorem{corollary}[theorem]{\bf{\textsc{Corollary}}}
\newtheorem{lemma}[theorem]{\bf{\textsc{Lemma}}}
\newtheorem{proposition}[theorem]{\bf{\textsc{Proposition}}}

\numberwithin{equation}{section}

\theoremstyle{definition}

\theoremstyle{definition}
\newtheorem{remark}[theorem]{\bf{\textsc{Remark}}}

\DeclareMathOperator{\Sym}{Sym}
\DeclareMathOperator{\ch}{ch}
\DeclareMathOperator{\proj}{pr}

\DeclareMathOperator{\Ad}{Ad}

\DeclareMathOperator{\sgn}{sgn}

\newcommand{\CC}{\mathbb{C}}
\newcommand{\RR}{\mathbb{R}}
\newcommand{\NN}{\mathbb{N}}
\newcommand{\sF}{\mathcal{F}}
\newcommand{\sH}{\mathcal{H}}
\newcommand{\sL}{\mathcal{L}}
\newcommand{\sN}{\mathcal{N}}
\newcommand{\sO}{\mathcal{O}}
\newcommand{\sR}{\mathcal{R}}
\newcommand{\frakg}{{\mathfrak g}}

\newcommand{\frakl}{{\mathfrak l}}
\newcommand{\frakn}{{\mathfrak n}}
\newcommand{\frakp}{{\mathfrak p}}
\newcommand{\frakt}{{\mathfrak t}}
\newcommand{\fraku}{{\mathfrak u}}
\newcommand{\frakz}{{\mathfrak z}}

\newcommand{\balpha}{{\boldsymbol\alpha}}
\newcommand{\bbeta}{{\boldsymbol\beta}}
\newcommand{\bgamma}{{\boldsymbol\gamma}}

\newcommand{\Set}[2]{\left\{#1\,\middle|\,#2\right\}}
\newcommand{\set}[2]{\{#1\,|\,#2\}}
\newcommand{\SymGp}{\mathfrak{S}}

\newcommand{\delbar}{\bar\partial}
\newcommand{\iCR}{{\bar D}}

\title[Hilbert series on generalized flag manifolds]
{Hilbert series of nearly holomorphic sections on generalized flag manifolds}
\author{Benjamin Schwarz} 

\date{\today}

\address{Benjamin Schwarz, Universit\"{a}t Paderborn,
Fakult\"{a}t f\"{u}r Elektrotechnik, Informatik und Mathematik,
Institut f\"{u}r Mathematik, Warburger Str. 100,
33098 Paderborn, Germany}
\email{bschwarz@math.upb.de}

\begin{document}
\begin{abstract}
Let $X=G/P$ be a complex flag manifold and $E\to X$ be a $G$-homogeneous holomorphic vector bundle. Fix a $U$-invariant Kähler metric on $X$ with $U\subseteq G$ maximal compact. We study the sheaf of nearly holomorphic sections and show that the space of global nearly holomorphic sections in $E$ coincides with the space of $U$-finite smooth sections in $E$. The degree of nearly holomorphic sections defines a $U$-invariant filtration on this space. Using sheaf cohomology, we determine in suitable cases the corresponding Hilbert series. It turns out that this is given in terms of Lusztig's $q$-analog of Kostant's weight multiplicity formula, and hence gives a new representation theoretic interpretation of these formulas.
\end{abstract}


\maketitle



\section*{Introduction}
In \cite{S86} Shimura introduced the notion of nearly holomorphic functions on Kähler manifolds in the course of his investigation of automorphic forms on classical bounded symmetric domains. 
It turned out to be an ``indispensable tool for the algebraicity of the critical values of certain zeta functions'' \cite{S94}. In this article, we consider a generalized notion of nearly holomorphic sections of holomorphic vector bundles on Kähler manifolds. The first natural question to answer is the existence of non-trivial nearly holomorphic sections. Unlike the bounded symmetric case, for compact Kähler manifolds this is a highly non-trivial question. We obtain the following positive result on generalized flag manifolds by means of representation theoretic methods.

\begin{theoremAlpha}\label{thm:A}
	Let $X=G/P$ be a generalized flag manifold, and $E=G\times^PE_o$ be the $G$-homogeneous vector 
	bundle associated to a simple $P$-module $E_o$. Let $U\subseteq G$ be maximal compact, and $h$ 
	be a $U$-invariant Kähler metric on $X$. Then, the space $\sN(X,E)$ of nearly holomorphic 
	sections in $E$ coincides with the space of $U$-finite smooth sections in $E$,
	\[
		\sN(X,E) = C^\infty(X,E)_{U-\text{finite}}.
	\]
	In particular, $\sN(X,E)$ is a dense subspace of $C^\infty(X,E)$ (with respect to uniform 
	convergence).
\end{theoremAlpha}

There is a notion of the degree of a nearly holomorphic section, which provides a filtration of $\sN(X,E)$, and the space $\sN^m(X,E)$ of nearly holomorphic sections of degree $\leq m$ is $U$-invariant and finite dimensional. The second goal of this article is to determine the corresponding Hilbert series,
\begin{align}\label{eq:HilbertSeries}
	\sH(\sN^\bullet_E,q):=\sum_{m\geq 0}\ch(\sN^m(X,E)/\sN^{m-1}(X,E))\,q^m,
\end{align}
where $\ch(M)$ denotes the formal character of a $U$-module $M$. As a first step, we show (in the setting of a general Kähler manifold) that the sheaf $\sN^m(E)$ associated to nearly holomorphic sections of degree $\leq m$ is a finitely generated locally free analytic sheaf and gives rise to an exact sequence of sheaves (see Theorem~\ref{thm:exactsequence}),
\begin{align}\label{eq:exactsequence}
	0\longrightarrow 
	\sN^{m-1}(E)\longrightarrow
	\sN^m(E)\longrightarrow
	\sO(E\otimes\Sym^m)\longrightarrow 0,
\end{align}
where $\Sym^m$ denotes the $m$'th symmetric power of the holomorphic tangent bundle. Using sheaf cohomology, this determines the Hilbert series up to higher cohomology terms. Applying cohomological vanishing results due to Broer \cite{Br93}, this yields the Hilbert series in the case of the full flag variety $X=G/B$, where $P=B$ is a Borel subgroup.

\begin{theoremAlpha}\label{thm:B}
	Let $X=G/B$ and $L_\mu=G\times^B \CC_\mu$ be the $G$-homogeneous line bundle associated to 
	$\mu\in\Lambda$. If $\mu$ is dominant, then
	\begin{align}\label{eq:HilbertSeriesFormula}
		\sH(\sN^\bullet_{L_\mu^*},q) = \sum_{\lambda\in\Lambda^+} m_\lambda^\mu(q)\,\ch V_\lambda^*,
	\end{align}
	where $m_\lambda^\mu(q)$ is Lusztig's $q$-analog of Kostant's weight multiplicity formula.
\end{theoremAlpha}
Here $L_\mu^*=L_{-\mu}$ denotes the dual line bundle, $\Lambda$ is the weight lattice associated to a torus $T\subseteq B$, $\Lambda^+\subseteq\Lambda$ is the set of dominant weights, and $V_\lambda^*$ denotes the dual of the $U$-module with highest weight $\lambda$. Lusztig's polynomials $m_\lambda^\mu(q)$ are defined \cite{Lu83} in purely combinatorial terms by
\begin{align}\label{eq:Lusztigqpoly}
	m_\lambda^\mu(q) := \sum_{w\in W}\sgn(w)\,\wp_q(w(\lambda+\rho)-\mu-\rho),
\end{align}
where $W$ denotes the Weyl group, $\rho$ is the half sum of positive roots, and $\wp_q$ is the $q$-analog of Kostant's partition function, i.e., the coefficient of $q^m$ in $\wp_q(\nu)$ is the number of ways to write $\nu\in\Lambda^+$ as the sum of precisely $m$ positive roots. These polynomials occur in various branches of representation theory. In the special case $\mu=0$ (corresponding to the trivial line bundle), the polynomials $m_\lambda^0(q)$ were first constructed, independently, by Hesselink \cite{He80} and Peterson \cite{Pe78}. They discovered that the polynomials $m_\lambda^0(q)$ are the coefficients of the Hilbert series corresponding to the (graded) coordinate ring of the nilpotent cone in the Lie algebra $\frakg$ of $G$. Prior to this, Kostant \cite{Kos63} determined this Hilbert series in terms of generalized exponents by an investigation of $G$-harmonic polynomials on $\frakg$. For general $\mu$, Lusztig and Kato proved that the $m_\lambda^\mu(q)$ are closely related to certain Kaszdan--Lusztig polynomials and hence encode deep combinatorial and geometric information, see \cite{Gu87}. We also like to mention that there is a close connection between the filtration of nearly holomorphic sections and the Brylinski--Kostant filtration of $U$-modules \cite{Gu87, HJ09,JLZ00}, which will be the topic of a subsequent article.

For a general flag manifold $X=G/P$ and a $G$-homogeneous vector bundle $E=G\times^P E_o$, we show that \eqref{eq:HilbertSeriesFormula} still holds true under the additional assumption that $E$ satisfies the cohomological vanishing condition 
\begin{align}\tag{V}
	H^k(X,\sO(E^*\otimes\Sym^m))=0\quad\text{for all}\quad k>0,\ m\geq 0,
\end{align}
where $\Sym^m$ denotes the $m$'th symmetric power of the holomorphic tangent bundle, and when Lusztig's polynomials are replaced by their parabolic versions, see Theorem~\ref{thm:HilbertSeries}. The vanishing condition is essentially a condition imposed on the highest weight $\mu$ of the simple $P$-module $E_o$. The list of known results given in Section~\ref{sec:vanishingcohomology} leads to the conjecture that \eqref{eq:vanishingcondition} holds for any dominant highest weight $\mu$. In the case of Hermitian symmetric spaces we confirm this by applying the results of \cite{Sc13a} which are based on a detailed analysis of nearly holomorphic sections in this setting.\\

Even though nearly holomorphic sections (and the corresponding higher order Cauchy--Riemann operators, see Section~\ref{subsec:CauchyRiemannOperators}) are naturally defined on general Kähler manifolds, they have been studied in detail only in quite restricted cases: for line bundles on bounded symmetric domains see \cite{EP96,PPZ90,S86,S87,S00,Zh00}, the case of the Riemann sphere is discussed in \cite{PZ93}, and general compact Hermitian symmetric spaces are investigated in \cite{Sc13b,Sc13a,S87}. Some general results concerning higher order Laplacian operators associated to Cauchy--Riemann operators of the trivial line bundle are obtained in \cite{EP96}. For further applications of nearly holomorphic sections and Cauchy--Riemann operators in representation theory, see \cite{KZ12,S90,Zh01,Zh02}.\\

We briefly outline the content of this article. Section~\ref{sec:NHS} is concerned with the theory of nearly holomorphic sections on general Kähler manifolds. We follow Engli\v{s} and Peetre \cite{EP96} and introduce nearly holomorphic sections as elements of the kernel of some (higher order) Cauchy--Riemann operator. In this way, it is clear that we obtain a sheaf of nearly holomorphic sections, and the main goal is to show exactness of the sequence \eqref{eq:exactsequence} in Theorem~\ref{thm:exactsequence}. From Section~\ref{sec:CRonFlagManifolds} onwards, we consider generalized flag manifolds $X=G/P$. We briefly recall the classification of $U$-invariant Kähler structures on $X$ and give an explicit description of the higher order Cauchy--Riemann operators. This is the key in proving Theorem~\ref{thm:A}. In Section~\ref{sec:HilbertSeries}, we use equivariant cohomology to prove the generalized version of Theorem~\ref{thm:B} mentioned above, see Theorem~\ref{thm:HilbertSeries}, and discuss the cohomological vanishing condition.\\



\section{Nearly holomorphic sections}\label{sec:NHS}
\subsection{Cauchy--Riemann operators}\label{subsec:CauchyRiemannOperators}
Let $(X,h)$ be a Kähler manifold with Kähler metric $h$, and let $E\to X$ be a holomorphic vector bundle. Let $T^{1,0}$ denote the holomorphic and $T^{0,1}$ denote the antiholomorphic tangent bundle on $X$. The metric induces a vector bundle isomorphism $(T^{0,1})^*\cong T^{1,0}$, which also induces an isomorphism between the spaces of smooth sections of $(T^{0,1})^*$ and of $T^{1,0}$. The composition of the $\delbar$-operator of $E$ and this isomorphism defines the \emph{Cauchy--Riemann operator} $\iCR$, 
\begin{align}
  \begin{tikzpicture}[baseline=-6mm]
    \node (C1) at (0,0) {$C^\infty(X,E)$};
    \node[right=of C1] (C2) {$C^\infty(X,E\otimes (T^{0,1})^*)$};
    \node[right=of C2] (C3) {$C^\infty(X,E\otimes T^{1,0})$.};
    \draw[->] (C1)--(C2) node [midway,above] {\small$\overline\partial$};
    \draw[->] (C2)--(C3) node [midway,above] {\small$h_*$};
		\draw[->,bend right=20] (C1) to node [below] {$\iCR$} (C3);
	\end{tikzpicture}
\end{align}
Since $E\otimes T^{1,0}$ is again a holomorphic vector bundle, iterates of the Cauchy--Riemann operator are defined in the obvious way. By abuse of notation, we simply write
\[
	\iCR^m :=\iCR\circ\cdots\circ\iCR:C^\infty(X,E)\to
	C^\infty(X,E\otimes (T^{1,0})^{\otimes m}).
\]
Since $h$ is a Kähler metric, and hence satisfies some symmetry properties, it turns out \cite[Lemma~2.0]{S86} that $\iCR^mf$ is actually a section in $E\otimes\Sym^m$, where $\Sym^m\subseteq (T^{1,0})^{\otimes m}$ denotes the holomorphic subbundle of symmetric tensors, so
\[
	\iCR^m\colon C^\infty(X,E)\to C^\infty(X,E\otimes\Sym^m).
\]
The kernels of these higher order Cauchy--Riemann operators define the filtered vector space of \emph{nearly holomorphic sections},
\[
	\sN^m(X,E):=\ker\iCR^{m+1},\quad
	\sN(X,E):=\bigcup_{m\geq 0}\sN^m(X,E).
\]
For $f\in \sN(X,E)$, the minimal $m$ such that $f\in\sN^m(X,E)$ is called the \emph{degree} of $f$.

\subsection{Sheaf of nearly holomorphic sections}
For a holomorphic vector bundle $E\to X$, let $\sO(E)$ and $C^\infty(E)$ denote the sheaves of  holomorphic and smooth sections in $E$, i.e., for open $U\subseteq X$,
\[
	\sO(E)(U) :=\sO(U,E),\quad C^\infty(E)(U) :=C^\infty(U,E).
\]
We may consider both sheaves as analytic sheaves, i.e., sheaves of $\sO_X$-modules, where $\sO_X$ is the structure sheaf of the complex manifold $X$. Then, $\iCR^{m+1}$ induces a homomorphism of analytic sheaves,
\[
	\iCR^{m+1}\colon C^\infty(E)\to C^\infty(E\otimes\Sym^{m+1}).
\]
Hence, the kernel is an analytic sheaf, which we call the \emph{sheaf of nearly holomorphic sections of degree $\leq m$} and denote by $\sN^m(E):=\ker\iCR^{m+1}$. We note that $\sN^0(E) = \sO(E)$, and for $m\geq 1$ there is a natural inclusion
\[
	\iota\colon\sN^{m-1}(E)\to\sN^m(E).
\]

\begin{theorem}\label{thm:exactsequence}
	For all $m\geq1$,
	\[
		0\longrightarrow 
		\sN^{m-1}(E)\stackrel{\iota}{\longrightarrow}
		\sN^m(E)\stackrel{\iCR^m}{\longrightarrow}
		\sO(E\otimes\Sym^m)\longrightarrow 0
	\]
	is a short exact sequence of locally free analytic sheaves on $X$. 
\end{theorem}
\begin{proof}
We first note that the sequence is well-defined, since for any open $U\subseteq X$, by definition $\iCR^m$ maps $\sN^m(U,E)$ into $C^\infty(U,E\otimes\Sym^m)$, and the image gets annihilated by $\iCR$, i.e., consists of holomorphic sections. For exactness, it is clear that $\iota$ is injective, and by definition $\sN^{m-1}(E)$ is the kernel of $\iCR^m$. It remains to show that $\iCR^m$ is a surjective homomorphism of sheaves. We show that each $x\in X$ has a neighborhood $U\subseteq X$ for which the equation $\iCR^mf=g$ is solvable for all $g\in\sO(U,E\otimes\Sym^m)$.

Recall the following local characterization of nearly holomorphic sections \cite{Sc13a}: Since $X$ is Kähler, the metric $h$ is locally given by a Kähler potential, so we can choose $U\subseteq X$ with local coordinates $(z^1,\ldots,z^n)$ and a smooth map $\Psi:U\to\RR$ such that the metric coefficients are given by 
\[
	h_{i\bar j} = \frac{\partial^2\Psi}{\partial\bar z^j\partial z^i}.
\]
For $1\leq\ell\leq n$, define the smooth map $Q_\ell:U\to\CC$ by
\[
	Q_\ell(z):=\frac{\partial\Psi}{\partial z^\ell}(z),\quad z\in U.
\]
Then \cite[Proposition~1.5]{Sc13a}, nearly holomorphic sections are characterized as polynomials in the functions $Q_\ell$ with holomorphic coefficients, i.e., $f\in C^\infty(U,E)$ is nearly holomorphic of degree $\leq m$ if and only if 
\[
	f(z) = \sum_{|I|\leq m} f_I(z)\,Q(z)^I\quad\text{with}\quad
	f_I\in\sO(U,E),\ z\in U.
\]
Here, we use standard multi-index notation, so for $I=(i_1,\ldots, i_n)\in\NN_0^n$ we set $|I| := \sum_{\ell=1}^n i_\ell$, and
\[
	Q(z)^I := \prod_{\ell=1}^n Q_\ell(z)^{i_\ell}.
\]
Moreover, the coefficients $f_I\in\sO(U,E)$ are uniquely determined by the choice of the Kähler potential. This shows that $\sN^m(U,E)$ is a free $\sO_X|_U$-module of finite rank, hence $\sN^m(E)$ is a locally free analytic sheaf. Now consider a holomorphic section $g\in\sO(U,E\otimes\Sym^m)$. Let $\partial_1,\ldots,\partial_n$ be the local frame for the holomorphic tangent bundle $T^{1,0}$ induced by the coordinate functions. For a multi-index $I\in\NN_0^n$ with $|I|=m$ set
\[
	\partial^I := \sum_{\sigma\in\SymGp_m}\sigma(
	\underbrace{\partial_1\otimes\cdots\otimes\partial_1}_{i_1\text{ times}}\otimes\cdots\otimes
	\underbrace{\partial_n\otimes\cdots\otimes\partial_n}_{i_n\text{ times}}),
\]
where $\SymGp_m$ denotes the symmetric group on $\{1,\ldots, m\}$ and 
\[
	\sigma(\partial_{j_1}\otimes\cdots\otimes\partial_{j_m})
	:= \partial_{j_{\sigma(1)}}\otimes\cdots\otimes\partial_{j_{\sigma(m)}}.
\]
Then $\partial^I$ is a section of $\Sym^m$ on $U$, and the set of all $\partial^I$ with $|I|=m$ forms a local frame for $\Sym^m$. Consider the decomposition 
\begin{align*}
	g = \sum_{|I|=m} g_I\otimes\partial^I
\end{align*}
with holomorphic $g_I\in\sO(U,E)$ and define
\[
	f(z):=\sum_{|I|=m} g_I(z)\cdot Q(z)^I.
\]
It is straightforward to check that 
\[
	\iCR(Q^I) = \sum_{k=1}^n i_k\,Q^{I-e_k}\,\partial_k,\quad
	\partial^J = \sum_{k=1}^n j_k\,\partial^{J-e_k}\otimes\partial_k,
\]
where $e_k$ denotes the $k$'th standard unit vector in $\RR^m$, so $I-e_k=(i_1,\ldots,i_k-1,\ldots,i_n)$ and likewise $J-e_k=(j_1,\ldots,j_k-1,\ldots,j_n)$.
By induction on $0\leq r\leq m$, it readily follows that
\[
	\iCR^rf(z) = \sum_{|I|= m-r,\,|J|=r}
					\tfrac{(I+J)!}{I!J!}\,g_{I+J}(z)\cdot Q(z)^I\otimes\partial^J(z)
\]
where $I!:=i_1!\cdots i_n!$. In particular, $\iCR^mf = g$.
\end{proof}

For any sheaf $\sF$ on $X$, let $H^k(X,\sF)$ denote its $k$'th cohomology. By Cartan--Serre theorem, coherence of $\sN^m(E)$ implies finiteness: 

\begin{corollary}\label{cor:finitedimensionality}
	If $X$ is compact, then $H^k(X,\sN^m(E))$ is finite dimensional for all $k,m\geq 0$.
\end{corollary}

In particular, since $\sN^m(X,E)=H^0(X,\sN^m(E))$, the space of global nearly holomorphic sections of degree $\leq m$ is finite dimensional if $X$ is compact, see \cite[Proposition~1.10]{Sc13a} for a proof of this fact using elliptic operator theory. We also note the following consequence of Theorem~\ref{thm:exactsequence}.

\begin{corollary}\label{cor:exactsequence}
	Let $m\geq 1$ and assume $H^1(X,\sO(E\otimes\Sym^\ell))=0$ for all $0\leq\ell<m$. Then,
	\[
		0\longrightarrow 
		\sN^{m-1}(X,E)\stackrel{\iota}{\longrightarrow}
		\sN^m(X,E)\stackrel{\iCR^m}{\longrightarrow}
		\sO(X,E\otimes\Sym^m)\longrightarrow 0
	\]
	is an exact sequence.
\end{corollary}
\begin{proof}
The short exact sequence in Theorem~\ref{thm:exactsequence} yields the long exact cohomology sequence
\begin{align*}
	0\to\,&\sN^{m-1}(X,E)\to\sN^m(X,E)\to\sO(X,E\otimes\Sym^m)\to\\
	&H^1(X,\sN^{m-1}(E))\to H^1(X,\sN^m(E))\to H^1(X,\sO(E\otimes\Sym^m))\to\cdots
\end{align*}
We have to show that $H^1(X,\sN^{m-1}(E))=0$. For $m=1$ this follows from $\sN^0(E) = \sO(E)$ and the assumption for $\ell=0$. Now let $m>1$. Since Theorem~\ref{thm:exactsequence} also applies for each $0\leq\ell<m$, we obtain the exact sequence
\begin{align*}
	\cdots\to H^1(X,\sN^{\ell-1}(E))\to
	H^1(X,\sN^\ell(E))\to\underbrace{H^1(X,\sO(E\otimes\Sym^\ell))}_{=0}\to\cdots,
\end{align*}
and it follows by induction on $\ell$ that $H^1(X,\sN^\ell(E))$ vanishes for all $0\leq\ell<m$, in particular $H^1(X,\sN^{m-1}(E))=0$.
\end{proof}

\begin{remark}\label{rmk:vanishingcohomology}
	If $X$ is a smooth projective variety (as in the case discussed below), there are natural 
	isomorphisms
	\[
		H^k((T^{1,0})^*,\sO_{\text{alg}}(\pi^*E)) \cong
		\bigoplus_{m=0}^\infty H^k(X,\sO(E\otimes\Sym^m))
	\]
	for all $k$, where $\pi:(T^{1,0})^*\to X$ is the holomorphic cotangent bundle and 
	$\sO_{\text{alg}}(\pi^*E)$ denotes the sheaf of algebraic sections in the pull-back bundle 
	$\pi^*E$. In some cases it is convenient to formulate cohomological vanishing conditions as in 
	Corollary~\ref{cor:exactsequence} (for all $m\geq0$) by vanishing conditions on the pull-back 
	bundle.
\end{remark}

\section{Cauchy--Riemann operators on flag manifolds}\label{sec:CRonFlagManifolds}
\subsection{Generalized flag manifolds}\label{sec:flagmanifolds}
We now turn to the setting of generalized flag manifolds, so $X=G/P$ where $G$ is a complex simple simply connected Lie group and $P\subseteq G$ is a parabolic subgroup. We also fix a Borel subgroup $B\subseteq P$ and a maximal torus $T\subseteq B$ with corresponding root data $\Delta\subseteq\Phi^+\subseteq\Phi:=\Phi(\frakg,\frakt)$, where $\Delta=\{\alpha_1,\ldots,\alpha_\ell\}$ are the simple and $\Phi^+$ the positive roots. Here and in the following, lower case Gothic letters denote Lie algebras of the corresponding Lie groups. We may also assume that $P$ is standard parabolic with respect to the given root data, so $P=P_\Pi$ for $\Pi\subseteq\Delta$. In more detail, we have the disjoint union $\Phi = \Xi_+\sqcup\Xi_0\sqcup\Xi_-$ with
\[
	\Xi_0 := \text{span}_\RR(\Pi)\cap\Phi,\quad
	\Xi_+ := \Phi^+\setminus\Xi_0,\quad
	\Xi_- := -\Xi_+,
\]
and setting
\[
	\frakl=\frakt \oplus \bigoplus_{\alpha\in\Xi_0}\frakg_\alpha,\quad
	\frakn^\pm=\bigoplus_{\alpha\in\Xi_{\pm}}\frakg_\alpha,
\]
where $\frakg_\alpha$ denotes the $\alpha-$root space of $\frakg$, we obtain
\[
	\frakg = \frakn^-\oplus\frakl\oplus\frakn^+, \quad
	\frakp=\frakl\oplus\frakn^+.
\]
The latter is the Levi decomposition of $\frakp$. Let $B$ denote the Killing form of $\frakg$. For each $\alpha\in\Phi$, define $H_\alpha$ by $\alpha(H) = B(H,H_\alpha)$ for all $H\in\frakt$, and choose a system $\set{x_\alpha\in\frakg_\alpha}{\alpha\in\Phi}$ such that $[x_\alpha,x_{-\alpha}]=H_\alpha$ and the structure constants $N_{\alpha,\beta}$ defined by 
\begin{align}\label{eq:structureconstants}
	[x_\alpha,x_\beta] = N_{\alpha,\beta}\,x_{\alpha+\beta}\quad(\alpha,\beta,\alpha+\beta\in\Phi)
\end{align}
are real valued, see \cite[VI.\,\S1]{Kna02}. For convenience, we also set $N_{\alpha,\beta}=0$ if $0\neq \alpha+\beta\notin\Phi$, and define $y_\alpha:=x_{-\alpha}$ for $\alpha\in\Phi$. Then the conjugate linear involution $\theta$ on $\frakg$, determined by $\theta H_\alpha = -H_\alpha$ and $\theta x_\alpha = -y_\alpha$ for $\alpha\in\Phi$, is a Cartan involution. Let $\fraku:=\frakg^\theta$ be the corresponding compact real form of $\frakg$, and $U\subseteq G$ be the connected compact Lie group with Lie algebra $\fraku$. Then, $X$ can be identified as a real manifold with the quotient $X\cong U/K$ where $K:=U\cap P$. We also note that $\theta\frakp = \frakn^-\oplus\frakl$.

\subsection{Invariant Kähler structures}
We briefly recall the classification of $U$-invariant Kähler metrics on $X=G/P$, for details we refer to \cite{A06}. The holomorphic and antiholo\-morphic tangent bundles are $G$-homogeneous vector bundles, which we identify as fiber products
\begin{align}\label{eq:tangentbdliso}
	T^{1,0}=G\times^P\frakg/\frakp,\quad T^{0,1}=G\times^P\frakg/\theta\frakp
\end{align}
with the $P$-action on the canonical fiber given by
\begin{align}\label{eq:paction}
	p.\underline v = \underline{\Ad_pv}\quad(\underline v\in\frakg/\frakp),\quad
	p.\underline w = \underline{\Ad_{\theta p} w}\quad(\underline w\in\frakg/\theta\frakp).
\end{align}
For simplicity, we underline elements to denote their equivalence classes. In the compact picture $X\cong U/K$, we may identify these bundles via the canonical $K$-equivariant isomorphisms $\frakn^-\cong\frakg/\frakp$ and $\frakn^+\cong\frakg/\theta\frakp$ with
\begin{align}\label{eq:cpttangentbdliso}
	T^{1,0}=U\times^K\frakn^-,\quad T^{0,1}=U\times^K\frakn^+.
\end{align}
Now, a $U$-invariant Kähler metric $h:T^{0,1}\times T^{1,0}\to\CC$ on $X$ is uniquely determined by its restriction to the base point $o=eK$, so $h$ uniquely corresponds to a $K$-invariant non-degenerate bilinear form
\[
	h_o:\frakn^+\times\frakn^-\to\CC,\ 
\]
such that $h_o(x,\theta x)>0$ for all $x\neq0$. There is an additional property of $h_o$ reflecting closedness of the Kähler form associated to the Kähler metric. It turns out that $h_o$ is of the form
\[
	h_o(x,y) = B(C_o,[x,y]),
\]
where $B$ is the Killing form of $\frakg$, and $C_o$ is an element of the cone
\[
	\frakz(\frakl)^+:=\set{H\in\frakz(\frakl)}
		{\alpha(H) > 0\text{ for all }\alpha\in\Delta\setminus\Pi}.
\]
Such an element $C_0$ is uniquely determined by $h_o$. Recall that $\frakz(\frakl)$ consists of elements $H\in\frakt$ satisfying $\alpha(H)=0$ for all $\alpha\in\Delta$. Conversely, any such form extends to a $U$-invariant Kähler metric, so $\frakz(\frakl)^+$ is the moduli space of $U$-invariant Kähler metrics on $X$. We like to note the coincidence that the $U$-orbit of $iC_o$ is an adjoint orbit in $\fraku$ isomorphic to $X\cong U/K$, and $h_o$ corresponds to the Kirillov--Kostant--Souriau form attached to this orbit. We finally note that for $\alpha,\beta\in\Xi_+$ and the root vectors $x_\alpha, y_\beta$ chosen in Section~\ref{sec:flagmanifolds},  we have
\begin{align}\label{eq:metriccoefficitents}
	h_o(x_\alpha,y_\beta) = c_\alpha\,\delta_{\alpha,\beta}\quad
	\text{with}\quad
	c_\alpha := \alpha(C_o)>0.
\end{align}
For later use, we also set $c_\alpha:=\alpha(C_o)$ for any $\alpha\in\frakt^*$.

\subsection{Cauchy--Riemann operators}
The goal of this section is to determine an explicit formula for the higher order Cauchy--Riemann operators. Let $E\to X$ be a $G$-homogeneous holomorphic vector bundle, so $E=G\times^P E_o$ where $E_o$ is a $P$-module. The formula in question will be given in the compact picture, i.e., smooth sections in $E$ are represented as smooth maps $f:U\to E_o$ satisfying the equivariance condition
\[
	f(uk) = k.f(u)\quad\text{for all}\quad u\in U,\,k\in K.
\]
We first recall a formula for the $\delbar$-operator on such sections, see e.g.\ \cite{GS69}. For $v\in\fraku$, let $\sR_vf$ denote the right derivative of $f$ along $v$,
\[
	\sR_vf(u):=\left.\frac{d}{dt} f(g e^{tv})\right|_{t=0},
\]
and for $x\in\frakg$, let $\sR_x^\CC f$ denote the complexified derivative defined by
\[
	\sR_x^\CC f = \sR_{v_1}f + i\,\sR_{v_2}f\quad\text{with}\quad x=v_1+iv_2,\quad 
	v_1,v_2\in\fraku.	
\]
Then, identifying the antiholomorphic cotangent bundle $(T^{0,1})^*$ with $U\times^K(\frakn^+)^*$ as in \eqref{eq:cpttangentbdliso}, the $\delbar$-operator yields a section $\delbar f$ in $E\otimes(T^{0,1})^*$ which is given by
\begin{align}\label{eq:delbarop}
	\delbar f(u)(x) = \sR_{x}^\CC f(u) + x.f(u).
\end{align}
The next step is to determine the first order Cauchy--Riemann operator. For $\alpha\in\Xi_+$, we set $\sR_\alpha:=\sR_{x_\alpha}^\CC$ with $x_\alpha$ given as in Section~\ref{sec:flagmanifolds}. Combining \eqref{eq:metriccoefficitents} and \eqref{eq:delbarop} it readily follows that the first order Cauchy--Riemann operator is given by
\begin{align}\label{eq:firstiCR}
	\iCR f(u) = \sum_{\alpha\in\Xi_+} c_\alpha^{-1}
							\left(\sR_\alpha f(u) + x_\alpha.f(u)\right)\otimes y_\alpha.
\end{align}
We note that if the canonical fiber $E_o$ of $E$ is a simple $P$-module, then the unipotent radical of $P$ acts trivially on $E_o$, so the term $x_\alpha.f(u)$ in \eqref{eq:firstiCR} vanishes. However, in general the canonical fiber of the tensor product bundle $E\otimes T^{1,0}$ is not simple, so for the iterates of the Cauchy--Riemann operator we also need the second term in \eqref{eq:firstiCR}.

For $m\geq1$ and a tuple $\balpha=(\alpha_1,\dots,\alpha_m)\in\Xi_+^m$, we set
\[
	y_\balpha:=y_{\alpha_1}\otimes\cdots\otimes y_{\alpha_m},
\]
and thus obtain a basis $\set{y_\balpha}{\balpha\in\Xi_+^m}$ of $(\frakn^-)^{\otimes m}$. In particular, a section $f\in C^\infty(X,E\otimes (T^{1,0})^{\otimes m})$ admits a decomposition
\[
	f(u) = \sum_{\balpha\in\Xi_+^m} f_\balpha(u)\otimes y_\balpha
\]
with smooth coefficient functions $f_\balpha\in C^\infty(U,E_o)$.

\begin{proposition}\label{prop:iCRoperatorOnFlagMfds}
	Let $X=G/P$, and let $E=G\times^P E_o$ be the $G$-homogeneous vector bundle associated to a 
	simple $P$-module $E_o$, and fix $m\in\NN$. In the compact picture, the action of the 
	Cauchy--Riemann operator $\iCR^m$ on $f\in C^\infty(X,E)$ is given by
	\begin{align}\label{eq:operatordecomposition}
		(\iCR^mf)(u) = \sum_{\balpha\in \Xi_+^m} (\iCR_\balpha f)(u)\otimes y_\balpha
	\end{align}
	with
	\begin{align}\label{eq:explicitiCR}
		(\iCR_\balpha f)(u) 
		 = \sum_{\sigma\in \mathfrak{S}_m} c_{\sigma(\balpha)}^{-1}\sR_{\sigma(\balpha)}f(u),
	\end{align}
	where $\sigma(\balpha) := (\alpha_{\sigma(1)},\ldots,\alpha_{\sigma(m)})$, and
	\begin{align*}
		c_{\bbeta}	&:= c_{\beta_1}\cdot c_{\beta_1+\beta_2}\cdots c_{\beta_1+\cdots+\beta_m},\qquad
		\sR_{\bbeta} :=\sR_{\beta_m}\sR_{\beta_{m-1}}\cdots\sR_{\beta_1}
	\end{align*}
	for any $\bbeta = (\beta_1,\ldots,\beta_m)\in \Xi_+^m$.
\end{proposition}

\begin{remark}
	For the setting of homogeneous vector bundles on flag manifolds, 
	Proposition~\ref{prop:iCRoperatorOnFlagMfds} gives an alternative proof of the 
	general fact that $\iCR^mf$ is indeed a section in the subbundle $E\otimes\Sym^m$ of 
	$E\otimes(T^{1,0})^{\otimes m}$.
\end{remark}

For the proof of Proposition~\ref{prop:iCRoperatorOnFlagMfds}, let $\iCR_\balpha$ be the operator defined by the decomposition \eqref{eq:operatordecomposition}, which acts on functions $f\in C^\infty(U,E_o)$. We then show that \eqref{eq:explicitiCR} holds.

\begin{lemma}\label{lem:iCRrecurrence}
	For $m>1$, the operators $\iCR_\balpha$ satisfy the recurrence relation
	\[
		\iCR_{(\alpha_1,\dots,\alpha_m)}
			= c_{\alpha_m}^{-1}\Big(\sR_{\alpha_m}\circ\iCR_{(\alpha_1,\ldots,\alpha_{m-1})}
			 +\sum_{k=1}^{m-1}N_{\alpha_k,\alpha_m}
			  \iCR_{(\alpha_1,\dots,\alpha_k+\alpha_m,\dots,\alpha_{m-1})}\Big),
	\]
	where $N_{\alpha,\beta}$ are the structure constants defined in \eqref{eq:structureconstants}.
\end{lemma}
\begin{proof}
For the proof it is more convenient to identify the canonical fiber $\frakn^-$ of $T^{1,0}$ with $\frakg/\frakp$ in order to have simple formulas for the $\frakp$-action on $\frakn^-$. As in \eqref{eq:paction} we underline elements to denote their equivalence classes. Applying \eqref{eq:firstiCR} to the section $\iCR^{m-1}f$ yields
\begin{align*}
	\iCR^mf &= \iCR(\iCR^{m-1}f)\\
		&= \sum_{\beta\in\Xi_+}c_\beta^{-1}\Big(\sR_\beta(\iCR^{m-1}f) + x_\beta.(\iCR^{m-1}f)\Big)
			\otimes\underline{y_\beta}\\
		&= \sum_{\balpha\in\Xi_+^{m-1},\,\beta\in\Xi_+} c_\beta^{-1}
			 \Big((\sR_\beta\circ\iCR_{\balpha}f)\otimes\underline{y_\balpha}
			 +\iCR_{\balpha}f\otimes \underline{x_\beta. y_\balpha}\Big)
			 \otimes\underline{y_\beta},
\end{align*}
where
\begin{align*}
	\underline{x_\beta.y_\balpha} &=
		\sum_{k=1}^{m-1} \underline{y_{\alpha_1}\otimes\cdots\otimes[x_\beta,y_{\alpha_k}]
		\otimes\cdots\otimes y_{\alpha_{m-1}}}\\
		&=\sum_{k=1}^{m-1} N_{\beta,-\alpha_k}\,
			\underline{y_{\alpha_1}\otimes\cdots\otimes y_{\alpha_k-\beta}
			\otimes\cdots\otimes y_{\alpha_{m-1}}},
\end{align*}
Setting $\alpha_k'=\alpha_k-\beta$, we obtain 
\[
	\sum_{\alpha_k\in\Xi_+} N_{\beta,-\alpha_k}\iCR_{\balpha}f\otimes\underline{y_{\alpha_k-\beta}}
	= \sum_{\alpha_k'\in\Xi_+} N_{\beta,-\beta-\alpha_k'}
		\iCR_{(\alpha_1,\dots,\alpha_k'+\beta,\dots,\alpha_{m-1})}f
		\otimes\underline{y_{\alpha_k'}}.
\]
We note that the summation over $\alpha_k'\in\Phi^+$ yields the same as the summation over $\alpha_k'\in-\beta+\Phi^+$, since either $\underline{y_{\alpha_k'}}$ or $N_{\beta,-\beta-\alpha_k'}$ vanishes. Finally, using $N_{\beta,-\beta-\alpha_k'} = N_{\alpha_k',\beta}$ we conclude that
\[
	\iCR^mf = \sum_{\balpha\in\Xi_+^{m-1},\,\beta\in\Xi_+} c_\beta^{-1}
			 \Big((\sR_\beta\circ\iCR_{\balpha}f)
			 +\sum_{k=1}^{m-1}N_{\alpha_k,\beta}\iCR_{\balpha+\beta\cdot e_k}f\Big)
			 \otimes\underline{y_\balpha}\otimes\underline{y_\beta},
\]
where $\balpha+\beta\cdot e_k=(\alpha_1,\cdots,\alpha_k+\beta,\cdots,\alpha_m)$.
\end{proof}

Using the recurrence relation of Lemma~\ref{lem:iCRrecurrence} we now prove Proposition~\ref{prop:iCRoperatorOnFlagMfds} by induction on $m$. For $m=1$, \eqref{eq:operatordecomposition} and \eqref{eq:explicitiCR} reduce to \eqref{eq:firstiCR}, since $E_o$ is assumed to be a simple $P$-module. Now let $m>1$. Due to Lemma~\ref{lem:iCRrecurrence} and by induction hypothesis, we have
\begin{align*}
	\iCR_{(\balpha,\beta)}f
		&= c_\beta^{-1}\Big((\sR_\beta\circ\iCR_{\balpha}f)
			 +\sum_{k=1}^{m-1}N_{\alpha_k,\beta}\iCR_{\balpha+\beta\cdot e_k}f\Big)\\
		&= \sum_{\sigma\in S_{m-1}}
					\Big(c_\beta^{-1}c_{\sigma(\balpha)}^{-1}\sR_\beta\sR_{\sigma(\balpha)}f
					+ c_\beta^{-1}\sum_{k=1}^{m-1} N_{\alpha_k,\beta}
						c_{\sigma(\balpha + \beta\cdot e_k)}^{-1}\sR_{\sigma(\balpha+\beta\cdot e_k)}f\Big).
\end{align*}
We next consider the sum over the second term. For fixed $1\leq k\leq m-1$, the symmetric group $\SymGp_{m-1}$ can be written as disjoint union
\[
	\SymGp_{m-1} = \coprod_{\ell=1}^{m-1}\,\Set{\sigma\in\SymGp_{m-1}}{\sigma(\ell)=k}.
\]
Recall that $\sR_\alpha=\sR_{x_\alpha}^\CC$ denotes the complexified right derivative of functions on $U$. Therefore, the map $x\mapsto\sR_x^\CC$ with $x\in\frakg$ is a Lie algebra homomorphism, so in particular, $N_{\alpha,\beta}\sR_{\alpha+\beta} = [\sR_\alpha,\sR_\beta]$. Therefore, the second term yields
\begin{align*}
	\sum_{\ell,k=1}^{m-1}&\sum_{\substack{\sigma\in\SymGp_{m-1}\\ \sigma(\ell)=k}}
		N_{\alpha_k,\beta}\,c_{\sigma(\balpha)+\beta\cdot e_\ell}^{-1}
		\sR_{\alpha_{\sigma(m-1)}}\cdots \sR_{\alpha_{\sigma(\ell)}+\beta}\cdots 
		\sR_{\alpha_{\sigma(1)}}f\\
	&= \sum_{\ell,k=1}^{m-1}\sum_{\substack{\sigma\in\SymGp_{m-1}\\ \sigma(\ell)=k}}
		c_{\sigma(\balpha)+\beta\cdot e_\ell}^{-1}
		\sR_{\alpha_{\sigma(m-1)}}\cdots[\sR_{\alpha_{\sigma(\ell)}},\sR_\beta]\cdots
		\sR_{\alpha_{\sigma(1)}}f\\
	&= \sum_{\ell=1}^{m-1}\sum_{\sigma\in\SymGp_{m-1}}
		c_{\sigma(\balpha)+\beta\cdot e_\ell}^{-1}
		\Big(\sR_{\alpha_{\sigma(m-1)}}\cdots \sR_{\alpha_{\sigma(\ell)}}\sR_\beta\cdots
		\sR_{\alpha_{\sigma(1)}}f\\
	&\hspace{4cm}	- \sR_{\alpha_{\sigma(m-1)}}\cdots \sR_\beta \sR_{\alpha_{\sigma(\ell)}}\cdots
		\sR_{\alpha_{\sigma(1)}}f\Big)\\[4mm]
	&= \sum_{\sigma\in S_{m-1}}c_{\sigma(\balpha)+\beta\cdot e_1}^{-1} 
			\sR_{\sigma(\balpha)}\sR_\beta f\\
	&\quad +\sum_{\ell=2}^{m-1}\sum_{\sigma\in\SymGp_{m-1}}
		 \Big(c_{\sigma(\balpha)+\beta\cdot e_{\ell}}^{-1}
		 	- c_{\sigma(\balpha)+\beta\cdot e_{\ell-1}}^{-1}\Big) \sR_{\alpha_{\sigma(m-1)}}\cdots 
		 	\sR_{\alpha_{\sigma(\ell)}}\sR_\beta\cdots \sR_{\alpha_{\sigma(1)}}f\\
	&\quad -\sum_{\sigma\in\SymGp_{m-1}} c_{\sigma(\balpha) + \beta\cdot e_{m-1}}^{-1}
		\sR_\beta \sR_{\sigma(\balpha)}f
\end{align*}
Collecting everything together we obtain
\begin{align*}
	\iCR_{(\balpha,\beta)}f 
		&= \sum_{\sigma\in\SymGp_{m-1}}c_\beta^{-1}\Big(c_{\sigma(\balpha)}^{-1} - 
		c_{\sigma(\balpha)+\beta\cdot e_{m-1}}^{-1}\Big)\sR_\beta\sR_{\sigma(\balpha)}f \\
	&+\sum_{\ell=2}^{m-1}\sum_{\sigma\in\SymGp_{m-1}}
		 c_\beta^{-1}\Big(c_{\sigma(\balpha)+\beta\cdot e_\ell}^{-1}
		 	- c_{\sigma(\balpha)+\beta\cdot e_{\ell-1}}^{-1}\Big)\sR_{\alpha_{\sigma(m-1)}}\cdots 
		 	\sR_{\alpha_{\sigma(\ell)}}\sR_\beta\cdots\sR_{\alpha_{\sigma(1)}}f\\[2mm]
	&+\sum_{\sigma\in\SymGp_{m-1}}c_\beta^{-1}c_{\sigma(\balpha)+\beta\cdot e_1}^{-1}
		 \sR_{\sigma(\balpha)}\sR_\beta f.
\end{align*}
It remains to show that the coefficients have the appropriate form. For convenience we set $\bgamma:=\sigma(\balpha)$. Recall that $c_\alpha = \alpha(C_0)$, so that $c_\alpha$ is additive in $\alpha$, i.e., $c_{\alpha+\beta}=c_\alpha+c_\beta$. It follows that
\begin{align*}
	c_\beta^{-1}(c_\bgamma^{-1} - c_{\bgamma+\beta\cdot e_{m-1}}^{-1})
	 &= c_\beta^{-1}c_\bgamma^{-1}\Big(1 -
			\frac{c_{\gamma_1+\cdots+\gamma_{m-1}}}
			 		 {c_{\gamma_1+\cdots+\gamma_{m-1}+\beta}}\Big)\\
	 &= c_\bgamma^{-1}c_{\gamma_1+\cdots+\gamma_{m-1}+\beta}^{-1}\\
	 &= c_{(\bgamma,\beta)}^{-1},\\[2mm]
	c_\beta^{-1}\Big(c_{\bgamma+\beta\cdot e_\ell}^{-1}
		 	- c_{\bgamma+\bbeta_{\ell-1}}^{-1}\Big)
		&= c_\beta^{-1}c_{\bgamma+\beta\cdot e_\ell}^{-1}\Big(
				1-\frac{c_{\gamma_1+\cdots+\gamma_{\ell-1}}}
					{c_{\gamma_1+\cdots+\gamma_{\ell-1}+\beta}}
				\Big)\\
		&= c_{\bgamma+\beta\cdot e_\ell}^{-1}
			c_{\gamma_1+\cdots+\gamma_{\ell-1}+\beta}^{-1}\\
		&= c_{(\gamma_1,\ldots,\gamma_\ell,\beta,\gamma_{\ell+1},\ldots,\gamma_{m-1})}^{-1},\\[2mm]
	c_\beta^{-1}c_{\bgamma+\beta\cdot e_1}^{-1} 
		&= c_{(\beta,\bgamma)}^{-1}.
\end{align*}
We thus conclude that
\[
	\iCR_{(\balpha,\beta)}f
	= \sum_{\sigma\in\SymGp_m} c_{\sigma(\balpha,\beta)}^{-1}\sR_{\sigma(\balpha,\beta)}f.
\]
This completes the proof of Proposition~\ref{prop:iCRoperatorOnFlagMfds}.\qed

\subsection{Proof of Theorem~\ref{thm:A}}
We will now prove that the space of nearly holomorphic sections in $E$ coincides with the space of $U$-finite smooth sections in $E$:
\begin{align}\label{eq:UFinEqNHS}
	\sN(X,E) = C^\infty(X,E)_{U-\text{finite}}.
\end{align}
By Corollary~\ref{cor:finitedimensionality}, each $\sN^m(X,E)$ is finite dimensional. Moreover, $\sN^m(X,E)$ is $U$-invariant, since Cauchy--Riemann operators are equivariant for bundle morphisms that restrict to isometries on the base manifold, see e.g.\ \cite{EP96}. For flag manifolds, this equivariance is also obvious from Proposition~\ref{prop:iCRoperatorOnFlagMfds}. It follows that $\sN^m(X,E)$ consists of $U$-finite vectors. Conversely, let $f\in C^\infty(X,E)$ be $U$-finite, and let $V:=\text{span}_\CC(U.f)$ denote the finite dimensional subspace generated by the $U$-action on $f$. In the compact picture, the action of $v\in\fraku$ on $F\in V$ is given (up to a sign) by left derivation, 
\[
	\sL_vF(u) = \left.\tfrac{d}{dt}F(e^{-tv}u)\right|_{t=0}.
\]
Let $\sL^\CC_xF$ denote the complexified action of $x\in\fraku_\CC=\frakg$. Since $V$ is finite dimensional, it follows that $\sL_{x_1}^\CC\cdots \sL_{x_m}^\CC F=0$ for all $F\in V$ and $x_1,\ldots,x_m\in\frakn^+$ if $m$ is sufficiently large $m\gg 0$. Since $\sL_v=-\sR_v$ at $u=e$, we conclude from Proposition~\ref{prop:iCRoperatorOnFlagMfds} that $(\iCR^mF)(e)=0$ for all $F\in V$ and sufficiently large $m\gg 0$. Now, the $U$-equivariance of the Cauchy--Riemann operators yield $\iCR^mf(u)=0$ for all $u\in U$, i.e., $f$ is nearly holomorphic. This proves \eqref{eq:UFinEqNHS}. Finally, recall the general fact on induced representations, that the space of $U$-finite smooth sections in $E$ is dense in $C^\infty(X,E)$ with respect to uniform convergence. This completes the proof of Theorem~\ref{thm:A}.\qed

\section{Hilbert series on flag manifolds}\label{sec:HilbertSeries}
We retain all definitons of the last section concerning generalized flag manifolds and the associated root data.

\subsection{$P$-version of Lusztig's polynomials}
Recall that $\Xi_+$ denotes the set of roots with root spaces contained in the nilradical $\frakn^+$ of $\frakp$. The $P$-version of Kostant's partition function is defined in a formal way by 
\[
	\prod_{\alpha\in\Xi_+} (1-qe^\alpha)^{-1}
		=: \sum_{\lambda\in\Lambda} \wp_{P,q}(\lambda)\,e^\lambda.
\]
In other words, the coefficient of $q^m$ in $\wp_{P,q}(\lambda)$ is the number of ways to write $\lambda\in\Lambda$ as the sum of precisely $m$ (not necessarily distinct) roots in $\Xi_+$. Then, Lusztig's polynomials \eqref{eq:Lusztigqpoly} generalize to
\begin{align*}
	m_\lambda^\mu(P;q) :=\sum_{w\in W} \sgn(w)\,\wp_{P,q}(w(\lambda+\rho) - \mu-\rho),
\end{align*}
where $\rho$ is the half sum of positive roots, and $W$ is the Weyl group of $(\frakg,\frakt)$. For $P=B$ and $q=1$, this is the classical Kostant multiplicity formula which determines the dimension of the $\mu$-weight space in a simple $U$-module of highest weight $\lambda$. The link between these purely combinatorial formulas and the geometry of the flag manifold $X=G/P$ is given below in Proposition~\ref{prop:EulerCharacter}. There, we determine the graded Euler character of a $G$-homogeneous vector bundle $E=G\times^PE_o$ defined by
\begin{align}\label{eq:EulerCharacter}
	\chi_{\text{gr}}(X,E) &:= \sum_{m\geq0} \chi(X,E\otimes\Sym^m)\,q^m,
\end{align}
where 
\[
	\chi(X,F)=\sum_{k\geq0} (-1)^k\ch(H^k(X,\sO(F)))
\]
denotes the usual Euler character of a $G$-homogeneous vector bundle $F=G\times^PF_0$. Here, $\ch(M)$ denotes the formal character of a finite dimensional $G$-module $M$.
%
Let $\Lambda_P^+$ denote the set of $P$-dominant weights of the weight lattice $\Lambda$, so $\lambda\in\Lambda_P^+$ if $\lambda(H_\alpha)>0$ for all $\alpha\in\Pi$. Then $\Lambda_P^+\supseteq\Lambda^+$, and $\Lambda_P^+$ parametrizes the set of simple $P$-modules by their highest weights. With our choice of positivity in the root system, it is more convenient to formulate the following statements by means of the dual vector bundle $E^*=G\times^PE_o^*$. For example, the Borel--Weil theorem states that $\sO(X,E^*)$ is non-trivial if and only if the highest weight $\mu$ of $E_o$ is dominant, and in this case, $\sO(X,E^*)\cong V_\mu^*$, where $V_\mu^*$ denotes the dual of the simple $U$-module $V_\mu$ the highest weight $\mu$, and all higher cohomology of $\sO(E^*)$ vanishes.

\begin{proposition}\label{prop:EulerCharacter}
	Let $E_\mu=G\times^P E_o$ be the $G$-homogeneous vector bundle on $X=G/P$ associated to the 
	simple $P$-module $E_o$ of highest weight $\mu\in\Lambda^+_P$. Then,
	\begin{align}\label{eq:gradedcharacter}
		\chi_{\text{gr}}(X,E_\mu^*)
			= \sum_{\lambda\in\Lambda^+} m_\lambda^\mu(P;q)\,\ch V_\lambda^*.
	\end{align}
\end{proposition}
\begin{proof}
For $P=B$, this is a well-known result of Brylinski \cite{Bry89}, based on the work of Hesselink \cite{He76}. The case of general parabolic subgroups is partially discussed more recently in \cite{Ha09}. Let $L_\nu=G\times^B\CC_\nu$ denote the line bundle on $G/B$ associated to $\nu\in\Lambda$, and let $\proj^*\Sym^m$ be the pull-back of $\Sym^m$ along the canonical projection $\proj:G/B\to G/P$. Then, $\proj^*\Sym^m = G\times^B\Sym^m(\frakg/\frakp)$, and applying Leray spectral sequences to the push forward of $L_{-\mu}$ along $\proj$ one shows \cite[Lemma~3.25]{Ha09} that
\[
	H^k(G/P,\sO(E_\mu^*\otimes\Sym^m))\cong H^k(G/B,\sO(L_{-\mu}\otimes \proj^*\Sym^m)).
\]
Since the $B$-module $\frakg/\frakp$ admits a filtration by $B$-stable subspaces such that consecutive quotients are the root spaces in $\frakn^-$ and hence enumerated by $\Xi_-=-\Xi_+$, it follows by additivity of the Euler character that
\[
	\chi_{\text{gr}}(X,E_\mu^*) = \sum_{\tau\in\Lambda}\wp_{P,q}(\tau-\mu)\,\chi(G/B,L_{-\tau}).
\]
Finally, the character $\chi(G/B,L_{-\tau})$ is determined by the Borel--Weil--Bott theorem, and comparison of the coefficients yields \eqref{eq:gradedcharacter}.
\end{proof}

\subsection{Cohomological vanishing condition}\label{sec:vanishingcohomology}
We now consider vector bundles $E\to X$ satisfying the following cohomological vanishing condition:
\begin{align}\label{eq:vanishingcondition}\tag{V}
	H^k(X,\sO(E^*\otimes\Sym^m))=0\quad\text{for all}\quad k>0,\ m\geq 0.
\end{align}
We note that this condition is equivalent to 
\[
	H^k((T^{1,0})^*,\sO_{\text{alg}}(\pi^*E^*))=0\quad\text{for all}\quad k>0,
\]
where $\pi:(T^{1,0})^*\to X$ is the holomorphic cotangent bundle considered as algebraic variety, see Remark~\ref{rmk:vanishingcohomology}. For $G$-homogeneous vector bundles $E_\mu=G\times^P E_o$ with simple $P$-module $E_o$, this vanishing condition is a condition on the highest weight $\mu\in\Lambda_P^+$ of $E_o$. It is known that \eqref{eq:vanishingcondition} holds in the following cases:
\begin{itemize}
	\item Any line bundle $E_\mu=G\times^P\CC_\mu$ with dominant $\mu\in\Lambda^+$.
				\cite{Br94}\vspace{2mm}
	\item $P=P_\Pi$ with $|\Pi|=1$ (sometimes called 'minimal parabolic') and dominant
				$\mu\in\Lambda^+$. \cite{Br94,Ha09}\vspace{2mm}
	\item Any $P$ and $\mu=\mu'-2\rho_P$ with $\mu'\in\Lambda^+$ $P$-regular (i.e., 
				$\mu'(H_\alpha)\neq 0$ for all $\alpha\in\Pi$), where $\rho_P$ is the half sum of roots in 
				$\Xi_+$. \cite{Ha09}\vspace{2mm}
	\item $G$ of type $A_n$, any $P$ and dominant regular $\mu\in\Lambda^+$. \cite{Ha09}\vspace{2mm}
\end{itemize}
The detailed analysis of nearly holomorphic sections in \cite{Sc13a} extends this list by the following result:

\begin{proposition}
	Let $X=G/P$ be Hermitian symmetric, i.e., the nilradical $\frakn^+$ of $\frakp$ is abelian. Then
	\eqref{eq:vanishingcondition} is satisfied for any vector bundle $E=G\times^P E_o$ associated 
	to a simple $P$-module $E_o$ with dominant highest weight $\mu\in\Lambda^+$.
\end{proposition}
\begin{proof}
Since $\frakn^+$ is abelian, it follows that $\frakn^+$ acts trivially on the canonical fiber $E_o^*\otimes\Sym^m(\frakn^-)$ of the vector bundle $E^*\otimes\Sym^m$, hence $E_o^*\otimes\Sym^m(\frakn^-)$ is completely reducible $P$-module. By \cite[Theorem~3.5]{Sc13a}, all lowest weights of simple $K$-modules contained in $E^*\otimes\Sym^m(\frakn^-)$ are antidominant (note that the Borel subgroup in \cite{Sc13a} used to define positivity in the root system is the opposite of the one used here). Due to Borel--Weil theorem this shows that all higher cohomology of the vector bundle $E^*\otimes\Sym^m$ vanishes.
\end{proof}
We expect \eqref{eq:vanishingcondition} to be true for arbitrary parabolic $P\subseteq G$ and dominant $\mu\in\Lambda^+$. In any case, assuming \eqref{eq:vanishingcondition}, we show that the Euler character \eqref{eq:EulerCharacter} of the dual bundle $E^*$ coincides with the Hilbert series 
\begin{align*}
	\sH(\sN^\bullet_{E^*},q)
		= \sum_{m\geq 0}\ch(\sN^m(X,E^*)/\sN^{m-1}(X,E^*))\,q^m,
\end{align*}
cf.\ \eqref{eq:HilbertSeries}, and thereby obtain our main result:

\begin{theorem}\label{thm:HilbertSeries}
	Let $X=G/P$ be a generalized flag variety, and let $E_\mu=G\times^PE_o$ be the $G$-homogeneous 
	vector bundle with simple $P$-module $E_o$ of highest weight $\mu\in\Lambda_P^+$. Assume that 
	$E_\mu$ satisfies \eqref{eq:vanishingcondition}. Then
	\[
		\sH(\sN^\bullet_{E_\mu^*}, q)
			= \sum_{\lambda\in\Lambda^+} m_\lambda^\mu(P;q)\,\ch V_\lambda^*.
	\]
\end{theorem}
\begin{proof}
By assumption \eqref{eq:vanishingcondition}, the graded Euler character reduces to
\[
	\chi_{\text{gr}}(X,E_\mu^*) = \sum_{m\geq0} \ch(\sO(X,E_\mu^*\otimes\Sym^m))\,q^m.
\]
Now it suffices to note that the exact sequence in Corollary~\ref{cor:exactsequence} is $U$-equivariant, so for $m>0$ we have isomorphisms of $U$-modules,
\[
	\sO(X,E_\mu^*\otimes\Sym^m)\cong\sN^m(X,E_\mu^*)/\sN^{m-1}(X,E_\mu^*),
\]
and $\sO(X,E_\mu^*)=\sN^0(X,E_\mu^*)$ holds by definition.
\end{proof}

\begin{remark}
	We note that Theorem~\ref{thm:HilbertSeries} provides another proof of Theorem~\ref{thm:A} for 
	vector bundles satisfying the vanishing condition \eqref{eq:vanishingcondition}, without 
	using the explicit form of the Cauchy--Riemann operators: The Hilbert series shows that the 
	space $\sN(X,E_\mu^*)$ of nearly holomorphic sections contains the simple $U$-module 
	$V_\lambda^*$ with multiplicity $m_\lambda^\mu(P;1)$. On the other hand, it is well known due to 
	Frobenius reciprocity and Kostant's branching theorem (see e.g.\ \cite{Kna02}), that the 
	multiplicity of $V_\lambda^*$ in the space of $U$-finite smooth sections in $E_\mu^*$ also 
	coincides with $m_\lambda^\mu(P;1)$. Therefore, any $U$-finite smooth section in $E^\mu$ must be 
	nearly holomorphic. The other inclusion follows from $U$-invariance and finite dimensionality 
	of $\sN^m(X,E_\mu^*)$ for all $m\geq0$.
\end{remark}

As a consequence of the proof of Theorem~\ref{thm:HilbertSeries} we also obtain:
\begin{corollary}\label{cor:SumDecomposition}
	Let $X=G/P$, $E_o$ be a simple $P$-modules, and let $E=G\times^PE_o$ satisfy 
	\eqref{eq:vanishingcondition}. Then there is a $U$-equivariant isomorphism
	\begin{align}\label{eq:Uisomorphism}
		C^\infty(X,E^*)_{U-\text{finite}} \cong\bigoplus_{m\geq0}\sO(X,E^*\otimes\Sym^m).
	\end{align}
\end{corollary}
Considering the holomorphic cotangent bundle $\pi:(T^{1,0})^*\to X$ as an algebraic variety, the right hand side of \eqref{eq:Uisomorphism} can also be identified with algebraic sections in the pull-back bundle $\pi^*E^*$ over $(T^{1,0})^*$, see also Remark~\ref{rmk:vanishingcohomology}.

\begin{remark}
	In the case of Hermitian symmetric spaces, the decomposition of the symmetric powers of 
	$\frakn^-$ into simple $K$-modules is well known due to Hua \cite{Hu63}, Kostant, and Schmid 
	\cite{Sc69}. Applying Corollary~\ref{cor:SumDecomposition} to the trivial line bundle 
	$E=X\times\CC$ then provides a connection between the Hua--Kostant--Schmid decomposition and the 
	Cartan--Helgason decomposition of $L^2(X)$ by means of the Borel--Weil theorem.
\end{remark}

\bibliographystyle{amsplain}
\providecommand{\bysame}{\leavevmode\hbox to3em{\hrulefill}\thinspace}
\providecommand{\MR}{\relax\ifhmode\unskip\space\fi MR }
\providecommand{\MRhref}[2]{%
  \href{http://www.ams.org/mathscinet-getitem?mr=#1}{#2}
}
\providecommand{\href}[2]{#2}


\end{document}